\documentclass[12pt]{article}
\usepackage[utf8]{inputenc}
\usepackage[T1]{fontenc}
\usepackage{amsmath}
\usepackage{amssymb}
\usepackage{amsthm}
\usepackage{geometry}
\usepackage{hyperref}
\usepackage{cite}

\theoremstyle{plain}
\newtheorem{theorem}{Theorem}[section]
\newtheorem{lemma}[theorem]{Lemma}

\begin{document}

\title{\textbf{On the Asymptotic Behavior of a Multiplicative Arithmetic Function Related to the Divisor Function Over Perfect Squares Integers Generated by Shifting}}
\author{Mihoub Bouderbala\textsuperscript{(1)}\\
\small\textsuperscript{(1)} Faculty of Matter Sciences and Computer Sciences, Department of Mathematics,\\
\small FIMA Laboratory, University of Djilali Bounaama, Khemis Miliana 44225, Algeria\\
\small mihoub75bouder@gmail.com}
\date{}
\maketitle

\begin{abstract}
\noindent Let $x$ be a real number satisfying $x \geq 2$. For any positive integer $n$, we define $s(n)$ as the smallest non-negative integer such that $n + s(n)$ is a perfect square. In this paper, we derive an asymptotic formula for the sum
\begin{equation*}
\sum_{n \leq x} D(n + s(n)),
\end{equation*}
where
\begin{equation*}
D(n) = \frac{\tau(n)}{2^{\omega(n)}}.
\end{equation*}
Here, $\tau(n)$ denotes the number of positive divisors of $n$, and $\omega(n)$ stands for the number of distinct prime factors of $n$.
\end{abstract}



\section{Introduction}

Throughout this paper, we let $\mathbb{N}$ denote the set of all positive integers. We write $a \mid b$ to indicate that $a$ divides $b$, for all positive integers $a$ and $b$. The notation $d \parallel n$ means that $d \mid n$ and $\gcd(d, n/d) = 1$. In particular, $p^{\alpha} \parallel n$ means that $p^{\alpha} \mid n$ but $p^{\alpha+1} \nmid n$, where $\alpha$ and $n$ are positive integers and $p$ denotes a prime number.

We recall that the arithmetic function $\tau(n)$ is defined by
\begin{equation}
\tau(n) := \sum_{d \mid n} 1,
\end{equation}
which counts the number of positive divisors of $n$, and
\begin{equation}
\omega(n) := \sum_{\substack{p \mid n \\ p \text{ prime}}} 1,
\end{equation}
which counts the number of distinct prime divisors of $n$.

Since the investigation of new properties of significant arithmetic functions has attracted considerable attention in recent years (see, e.g., \cite{Bouderbala2023, Indlekofer2002, Stepanauskas1997}), we focus in this paper on a novel aspect of the multiplicative arithmetic function defined by
\begin{equation}
D(n) := \frac{\tau(n)}{2^{\omega(n)}}.
\end{equation}

It is important to note that for any positive integer $n$,
\begin{equation}
2^{\omega(n)} = \sum_{d \parallel n} 1,
\end{equation}
where the sum runs over all unitary divisors $d$ of $n$. Consequently, $D(n)$ is multiplicative (see, e.g., \cite{Bouderbala2023} for further details). Moreover, for every prime number $p$ and every positive integer $m$, we have
\begin{equation}
D(p^m) = \frac{m+1}{2}.
\end{equation}

Using identity (1), it follows that for any $n \in \mathbb{N}$ with prime factorization $n = \prod_{p^{\alpha} \parallel n} p^{\alpha}$, we have
\begin{equation}
D(n) = \prod_{p^{\alpha} \parallel n} \frac{\alpha+1}{2}.
\end{equation}

In addition, we define the square-completion shift function as follows: for any positive integer $n$, let $s(n)$ be the smallest non-negative integer such that $n + s(n)$ is a perfect square. In other words, for every $n \in \mathbb{N}$, there exists a unique smallest integer $m \in \mathbb{N}$ such that
\begin{equation}
n + s(n) = m^2.
\end{equation}

For further clarity, we list some initial values of the function $s(n)$:
\begin{equation*}
s(1) = 0, \quad s(2) = 2, \quad s(3) = 1, \quad s(4) = 0, \quad s(5) = 4, \quad s(6) = 3, \quad s(7) = 2, \quad s(8) = 1, \ldots
\end{equation*}

In this paper, we employ analytic methods to investigate the asymptotic behavior of the arithmetic function $D(n)$ evaluated at these shifted square-completed integers. More precisely, we consider the summatory function
\begin{equation*}
\sum_{n \leq x} D(n + s(n)),
\end{equation*}
where $x \geq 2$ is a real number, and we establish a sharp asymptotic formula for this sum.

\section{Main Result}

We will prove the following theorem.

\begin{theorem}
Let $x \geq 2$ be a real number. Then we have the asymptotic formula
\begin{equation}
\sum_{n \leq x} D(n + s(n)) = \frac{C_1}{2} x \log x + \left(\left(2\gamma - \frac{1}{2}\right)C_1 + C_2\right) x + O(x^{3/4+\varepsilon}),
\end{equation}
where $\varepsilon > 0$ is an arbitrary fixed constant, $\gamma$ denotes the Euler--Mascheroni constant, and the constants $C_1$ and $C_2$ are defined by
\begin{equation}
C_1 := \prod_{p} \left(1 - \frac{1}{2p} + \frac{1}{2p^2}\right), \quad C_2 := P'(1),
\end{equation}
with $P(s) = \prod_{p} \left(1 - \frac{1}{2p^s} + \frac{1}{2p^{2s}}\right)$.
\end{theorem}

\subsection{Proof of Theorem 2.1}

To establish the main result, we first require the following auxiliary lemma, which provides the asymptotic behavior of the sum involving $D(n^2)$.

\begin{lemma}
For any real number $x \geq 2$, we have
\begin{equation}
\sum_{n \leq x} D(n^2) = C_1 x \log x + \left((2\gamma - 1)C_1 + C_2\right) x + O(x^{1/2+\varepsilon}),
\end{equation}
where $C_1$, $C_2$, $\gamma$, and $\varepsilon > 0$ are as in Theorem 2.1.
\end{lemma}

\begin{proof}
The arithmetic function is defined by
\begin{equation*}
D(n) := \frac{\tau(n)}{2^{\omega(n)}},
\end{equation*}
where $\tau(n)$ denotes the number of positive divisors of $n$ and $\omega(n)$ the number of its distinct prime divisors. We consider the Dirichlet series associated with $D(n^2)$, namely
\begin{equation*}
f(s) := \sum_{n=1}^{\infty} \frac{D(n^2)}{n^s}, \quad s = \sigma + it.
\end{equation*}

Since the function $n \mapsto D(n^2)$ is multiplicative, its Dirichlet series admits an Euler product (\cite[p. 320]{Apostol1976}):
\begin{equation*}
f(s) = \prod_{p} \left( \sum_{m=0}^{\infty} \frac{D(p^{2m})}{p^{ms}} \right).
\end{equation*}

Indeed, for $n = p^m$, we have $n^2 = p^{2m}$, so the term corresponding to $n = p^m$ in the Dirichlet series is $D(p^{2m})/p^{ms}$. Thus, the exponent in the denominator is $m$ (since $n = p^m$), while the argument of $D$ is $p^{2m}$.

For a fixed prime $p$, the local Euler factor is therefore
\begin{equation*}
f_p(s) := \sum_{m=0}^{\infty} \frac{D(p^{2m})}{p^{ms}}.
\end{equation*}

We now compute $D(p^{2m})$. Recall that
\begin{equation*}
\tau(p^{2m}) = 2m + 1, \quad \omega(p^{2m}) = \begin{cases} 0 & \text{if } m = 0, \\ 1 & \text{if } m \geq 1. \end{cases}
\end{equation*}

Hence,
\begin{equation*}
D(p^{2m}) = \frac{\tau(p^{2m})}{2^{\omega(p^{2m})}} = \begin{cases} 1 & \text{if } m = 0, \\ \displaystyle\frac{2m+1}{2} & \text{if } m \geq 1. \end{cases}
\end{equation*}

It follows that
\begin{equation*}
f_p(s) = 1 + \sum_{m=1}^{\infty} \frac{2m+1}{2 p^{ms}} = 1 + \frac{1}{2} \sum_{m=1}^{\infty} (2m+1) p^{-ms}.
\end{equation*}

Set $x = p^{-s}$ (so $|x| < 1$ for $\Re(s) > 0$). Using the classical identity
\begin{equation*}
\sum_{m=0}^{\infty} (2m+1) x^m = \frac{1+x}{(1-x)^2},
\end{equation*}
we obtain
\begin{equation*}
\sum_{m=1}^{\infty} (2m+1) x^m = \frac{1+x}{(1-x)^2} - 1 = \frac{3x - x^2}{(1-x)^2}.
\end{equation*}

Therefore,
\begin{equation*}
f_p(s) = 1 + \frac{1}{2} \cdot \frac{3x - x^2}{(1-x)^2} = \frac{2(1-x)^2 + 3x - x^2}{2(1-x)^2}.
\end{equation*}

By expanding $(1-x)^2$, the numerator becomes
\begin{equation*}
2(1-x)^2 + 3x - x^2 = 2 - x + x^2.
\end{equation*}

Hence,
\begin{equation*}
f_p(s) = \frac{2 - x + x^2}{2(1-x)^2} = \frac{2 - p^{-s} + p^{-2s}}{2(1 - p^{-s})^2}.
\end{equation*}

The full Dirichlet series is thus
\begin{equation*}
f(s) = \prod_{p} f_p(s) = \prod_{p} \frac{2 - p^{-s} + p^{-2s}}{2(1 - p^{-s})^2}.
\end{equation*}

Since $\zeta(s) = \prod_{p} (1 - p^{-s})^{-1}$, we have $\zeta^2(s) = \prod_{p} (1 - p^{-s})^{-2}$. Factoring this out yields
\begin{equation}
f(s) = \zeta^2(s) \prod_{p} \frac{2 - p^{-s} + p^{-2s}}{2} = \zeta^2(s) \prod_{p} \left(1 - \frac{1}{2p^s} + \frac{1}{2p^{2s}}\right), \quad \Re(s) > 1.
\end{equation}

Since $D(n^2) \leq \tau(n^2) \ll_{\varepsilon} n^{\varepsilon}$ for every $\varepsilon > 0$ (see, e.g., \cite[p. 69]{Montgomery2007}), the Dirichlet series $f(s)$ converges absolutely for $\Re(s) > 1$.

Then, from this identity and by applying Perron's formula (see, e.g., \cite[p. 147]{Tenenbaum1995}), we obtain the following expression, valid for all $x \geq 2$ and $T \geq 1$:
\begin{equation}
\sum_{n \leq x} D(n^2) = \frac{1}{2\pi i} \int_{3/2-iT}^{3/2+iT} f(s) \frac{x^s}{s} \, ds + O\left(\frac{x^{3/2+\varepsilon}}{T}\right),
\end{equation}
where $\varepsilon > 0$ is an arbitrary fixed constant.

To evaluate the main term, we shift the line of integration from $\Re(s) = 3/2$ to $\Re(s) = 1/2$. Set
\begin{equation*}
F(s) := f(s) \frac{x^s}{s} = \zeta^2(s) P(s) \frac{x^s}{s},
\end{equation*}
where
\begin{equation*}
P(s) := \prod_{p} \left(1 - \frac{1}{2p^s} + \frac{1}{2p^{2s}}\right).
\end{equation*}

The function $F(s)$ is meromorphic in the half-plane $\Re(s) > 1/2$ and possesses a double pole at $s = 1$, arising from $\zeta^2(s)$. By the residue theorem, we have
\begin{equation*}
\frac{1}{2\pi i} \left( \int_{3/2-iT}^{3/2+iT} + \int_{3/2+iT}^{1/2+iT} + \int_{1/2+iT}^{1/2-iT} + \int_{1/2-iT}^{3/2-iT} \right) f(s) \frac{x^s}{s} \, ds = \text{Res}(F(s), 1).
\end{equation*}

Using the classical Laurent expansion of the Riemann zeta function around $s = 1$,
\begin{equation*}
\zeta(s) = \frac{1}{s-1} + \gamma + O(s-1),
\end{equation*}
where $\gamma$ denotes the Euler--Mascheroni constant, we deduce
\begin{equation*}
\zeta^2(s) = \frac{1}{(s-1)^2} + \frac{2\gamma}{s-1} + O(1).
\end{equation*}

Since $P(s)$ is holomorphic and non-vanishing at $s = 1$, we write the local expansion
\begin{equation*}
P(s) = P(1) + P'(1)(s-1) + O((s-1)^2).
\end{equation*}

A standard computation of the residue of a double pole then yields
\begin{equation*}
\text{Res}(F(s), 1) = x \left( P(1)(\log x + 2\gamma - 1) + P'(1) \right).
\end{equation*}

Consequently,
\begin{equation}
\frac{1}{2\pi i} \left( \int_{3/2-iT}^{3/2+iT} + \int_{3/2+iT}^{1/2+iT} + \int_{1/2+iT}^{1/2-iT} + \int_{1/2-iT}^{3/2-iT} \right) f(s) \frac{x^s}{s} \, ds = x \left( C_1(\log x + 2\gamma - 1) + C_2 \right),
\end{equation}
where
\begin{equation*}
C_1 := P(1) = \prod_{p} \left(1 - \frac{1}{2p} + \frac{1}{2p^2}\right), \quad C_2 := P'(1).
\end{equation*}

Now choose $T = x$. The integrals over the horizontal segments satisfy
\begin{equation}
\left| \frac{1}{2\pi i} \left( \int_{1/2-iT}^{3/2-iT} + \int_{3/2+iT}^{1/2+iT} \right) f(s) \frac{x^s}{s} \, ds \right| \ll \int_{1/2}^{3/2} \left| \zeta^2(\sigma + iT) P(\sigma + iT) \right| \frac{x^{\sigma}}{T} \, d\sigma \ll x^{1/2+\varepsilon},
\end{equation}
using the classical bound $\zeta(\sigma + it) \ll |t|^{\varepsilon}$ for $\sigma \geq 1/2$ and $|t| \geq 1$, together with the absolute convergence of $P(s)$ in $\Re(s) > 1/2$.

Similarly, for the vertical integral on $\Re(s) = 1/2$, we have
\begin{equation}
\left| \frac{1}{2\pi i} \int_{1/2-iT}^{1/2+iT} f(s) \frac{x^s}{s} \, ds \right| \ll \int_{-T}^{T} \left| \zeta^2\left(\frac{1}{2} + it\right) P\left(\frac{1}{2} + it\right) \right| \frac{x^{1/2}}{|t| + 1} \, dt \ll x^{1/2+\varepsilon}.
\end{equation}

Finally, combining estimates (2), (3), (4), and (5), we conclude that
\begin{equation}
\sum_{n \leq x} D(n^2) = x \left( C_1(\log x + 2\gamma - 1) + C_2 \right) + O(x^{1/2+\varepsilon}).
\end{equation}

This completes the proof.
\end{proof}

\textbf{Proof of Theorem 2.1.} Let $x \geq 2$ be a real number. There exists a unique integer $N \geq 1$ such that
\begin{equation*}
N^2 \leq x < (N+1)^2.
\end{equation*}

Using this decomposition, we split the sum as
\begin{equation}
\sum_{n \leq x} D(n + s(n)) = \sum_{n < N^2} D(n + s(n)) + \sum_{N^2 \leq n \leq x} D(n + s(n)),
\end{equation}
where the function $s(n)$ is defined implicitly by the relation $n + s(n) = m^2$ for the smallest integer $m$ with $m^2 \geq n$. In particular, for $k^2 \leq n < (k+1)^2$, one has $n + s(n) = (k+1)^2$.

We first estimate the tail sum. Since $D(m) \ll_{\varepsilon} m^{\varepsilon}$ for any $\varepsilon > 0$, we obtain
\begin{equation}
\sum_{N^2 \leq n \leq x} D(n + s(n)) \ll_{\varepsilon} \sum_{N^2 \leq n \leq x} (n + s(n))^{\varepsilon} \ll_{\varepsilon} \sum_{N^2 \leq n < (N+1)^2} x^{\varepsilon} \ll x^{\varepsilon} \cdot \left((N+1)^2 - N^2\right) \ll x^{1/2+\varepsilon}.
\end{equation}

We now treat the main part of the sum. Using the definition of $s(n)$, for each integer $k \geq 1$ and $n$ such that $k^2 \leq n < (k+1)^2$, we have $n + s(n) = (k+1)^2$. Hence,
\begin{equation*}
\sum_{n < N^2} D(n + s(n)) = \sum_{k=1}^{N-1} \sum_{k^2 \leq n < (k+1)^2} D((k+1)^2) = \sum_{k=1}^{N-1} (2k+1) D((k+1)^2).
\end{equation*}

Shifting the index ($m = k+1$) yields
\begin{equation*}
\sum_{n < N^2} D(n + s(n)) = \sum_{m=2}^{N} (2m-1) D(m^2) = 2 \sum_{m=1}^{N} m D(m^2) - \sum_{m=1}^{N} D(m^2) + O(1).
\end{equation*}

By the bound $\sum_{m \leq N} D(m^2) \ll N^{1+\varepsilon}$ (which follows from the known asymptotic of this sum), the second term is $O(N^{1+\varepsilon}) = O(x^{1/2+\varepsilon})$. Thus,
\begin{equation}
\sum_{n < N^2} D(n + s(n)) = 2 \sum_{m=1}^{N} m D(m^2) + O(x^{1/2+\varepsilon}).
\end{equation}

To evaluate $\sum_{m \leq N} m D(m^2)$, define
\begin{equation*}
S(t) := \sum_{m \leq t} D(m^2).
\end{equation*}

By partial summation (see, e.g., \cite[p. 9]{Bordelles2020}), we have
\begin{equation*}
\sum_{m \leq N} m D(m^2) = N S(N) - \int_{1}^{N} S(t) \, dt.
\end{equation*}

From the previous section, we know that
\begin{equation*}
S(t) = C_1 t \log t + C_1(2\gamma - 1)t + C_2 t + O(t^{1/2+\varepsilon}),
\end{equation*}
where
\begin{equation*}
C_1 = \prod_{p} \left(1 - \frac{1}{2p} + \frac{1}{2p^2}\right), \quad C_2 = P'(1),
\end{equation*}
and $P(s) = \prod_{p} \left(1 - \frac{1}{2p^s} + \frac{1}{2p^{2s}}\right)$.

Substituting this expansion, we obtain
\begin{align}
\sum_{m \leq N} m D(m^2) &= N \left[ C_1 N \log N + C_1(2\gamma - 1)N + C_2 N + O(N^{1/2+\varepsilon}) \right] \nonumber \\
&\quad - \int_{1}^{N} \left[ C_1 t \log t + C_1(2\gamma - 1)t + C_2 t + O(t^{1/2+\varepsilon}) \right] dt \nonumber \\
&= C_1 N^2 \log N + C_1(2\gamma - 1)N^2 + C_2 N^2 \nonumber \\
&\quad - \left[ \frac{C_1}{2} N^2 \log N - \frac{C_1}{4} N^2 + \frac{C_1(2\gamma - 1)}{2} N^2 + \frac{C_2}{2} N^2 \right] + O(N^{3/2+\varepsilon}) \nonumber \\
&= \frac{C_1}{2} N^2 \log N + \left( \frac{2\gamma - 1}{2} C_1 + \frac{C_2}{2} \right) N^2 + O(N^{3/2+\varepsilon}).
\end{align}

Inserting (9) into (8) gives
\begin{equation}
\sum_{n < N^2} D(n + s(n)) = C_1 N^2 \log N + \left( \left(2\gamma - \frac{1}{2}\right)C_1 + C_2 \right) N^2 + O(N^{3/2+\varepsilon}).
\end{equation}

Finally, recall that $N = \lfloor \sqrt{x} \rfloor$, so $N = \sqrt{x} + O(1)$ and $N^2 = x + O(\sqrt{x})$. Using the expansion $\log N = \frac{1}{2} \log x + O(x^{-1/2})$, we obtain
\begin{equation*}
N^2 \log N = \frac{1}{2} x \log x + O(x^{1/2} \log x).
\end{equation*}

Combining this with (6), (7), and (10), we conclude that
\begin{equation*}
\sum_{n \leq x} D(n + s(n)) = \frac{C_1}{2} x \log x + \left( \left(2\gamma - \frac{1}{2}\right)C_1 + C_2 \right) x + O(x^{3/4+\varepsilon}).
\end{equation*}

This completes the proof of the theorem.

\end{document}